\documentclass{article}
\usepackage{amsfonts,amsmath,amsthm,amssymb}
\usepackage{graphics, epsfig}
\usepackage{color}
\usepackage{appendix}
\usepackage[margin=1.5in]{geometry}

\let\TeXchi\chi
\newbox\chibox
\setbox0 \hbox{\mathsurround0pt $\TeXchi$}
\setbox\chibox \hbox{\raise\dp0 \box 0 }
\def\chi{\copy\chibox}
\begin{document}
\newtheorem{proposition}{Proposition}[section]
\newtheorem{theorem}{Theorem}[section]
\newtheorem{lemma}{Lemma}[section]
\newtheorem{corollary}{Corollary}[section]
\newtheorem{remark}{Remark}[section]
\renewcommand{\thesection}{\arabic{section}}
\renewcommand{\theequation}{\thesection.\arabic{equation}}
\renewcommand{\thetheorem}{\thesection.\arabic{theorem}}
\numberwithin{equation}{section}
\numberwithin{theorem}{section}
\numberwithin{proposition}{section}
\numberwithin{lemma}{section}
\numberwithin{remark}{section}
\setcounter{secnumdepth}{3}
\newcommand{\cl}{\centerline}
\newcommand{\sms}{\smallskip}
\newcommand{\ms}{\medskip}
\newcommand{\bs}{\bigskip}
\newcommand{\noi}{\noindent}
\newcommand{\itl}[1]{\textit{#1}}
\newcommand{\blf}[1]{\textbf{#1}}
\newcommand{\dsty}{\displaystyle}
\newcommand{\txty}{\textstyle}
\newcommand{\ssty}{\scriptstyle}
\newcommand{\tty}{\texttt}


\newcommand\Par{\mathhexbox278\,}


\newcommand{\al}{\alpha}
\newcommand{\Al}{\Alpha}
\newcommand{\be}{\beta}
\newcommand{\Be}{\Beta}
\newcommand{\Gm}{\Gamma}
\newcommand{\gm}{\gamma}
\newcommand{\dl}{\delta}
\newcommand{\Dl}{\Delta}
\newcommand{\lm}{\lambda}
\newcommand{\Lm}{\Lambda}
\newcommand{\kp}{\kappa}
\newcommand{\varep}{\varepsilon}
\newcommand{\eps}{\epsilon}
\newcommand{\vp}{\varphi}
\newcommand{\sig}{\sigma}
\newcommand{\Sig}{\Sigma}
\newcommand{\om}{\omega}
\newcommand{\Om}{\Omega}
\newcommand{\uom}{\mbox{\boldmath$\omega$}}
\newcommand{\btau}{\mbox{\boldmath$\tau$}}
\newcommand{\bnu}{\mbox{\boldmath$\nu$}}
\newcommand{\up}{\upsilon}
\newcommand{\z}{\zeta}


\newcommand{\df}[1]{\buildrel\mbox{\small def}\over{#1}}
\newcommand{\op}[1]{\buildrel\mbox{\tiny o}\over{#1}}
\newcommand{\db}{\prime\prime}
\newcommand{\bsl}{\backslash}
\newcommand{\lb}{\lbrack\!\lbrack}
\newcommand{\rb}{\rbrack\!\rbrack}
\newcommand\la{\langle}
\newcommand\ra{\rangle}
\newcommand{\ev}{\equiv}
\newcommand{\nev}{\not\equiv}
\newcommand{\nn}{\mathbb{N}}
\newcommand{\qq}{\mathbb{Q}}
\newcommand{\zz}{\mathbb{Z}}
\newcommand{\rr}{\mathbb{R}}
\newcommand{\rn}{\rr^N}
\newcommand{\cc}{\mathbb{C}}
\newcommand{\id}{\mathbb{I}}
\newcommand{\bo}{\mathbb{O}}

\newcommand{\amsb}[1]{\mathbb{#1}}
\newcommand{\mcl}[1]{\mathcal{#1}}
\newcommand{\bl}[1]{\mathbf{#1}}
\newcommand{\ov}[1]{\overline{#1}}
\newcommand{\wt}[1]{\widetilde{#1}}
\newcommand{\wh}[1]{\widehat{#1}}

\newcommand{\llra}{\leftrightarrow}
\newcommand{\lra}{\longrightarrow}
\newcommand{\LLR}{\Longleftrightarrow}
\newcommand{\LRA}{\Longrightarrow}
\newcommand{\LLA}{\Longleftarrow}


\newcommand{\bbox}{\vrule height.6em width.6em 
depth0em} 
\newcommand{\os}{\vbox{\hrule \hbox{\vrule 
height.6em depth0pt 
\hskip.6em \vrule height.6em depth0em}
\hrule}} 


\newcommand{\dvg}{\operatorname{div}}
\newcommand{\curl}{\operatorname{curl}}
\newcommand{\supp}{\operatorname{supp}}
\newcommand{\essup}{\operatornamewithlimits{ess\,sup}}
\newcommand{\essinf}{\operatornamewithlimits{ess\,inf}}
\newcommand{\essosc}{\operatornamewithlimits{ess\,osc}}
\newcommand{\osc}{\operatornamewithlimits{osc}}
\newcommand{\sign}{\operatorname{sign}}
\newcommand{\loc}{\operatorname{loc}}
\newcommand{\diam}{\operatorname{diam}}
\newcommand{\dist}{\operatorname{dist}}
\newcommand{\card}{\operatorname{card}}
\newcommand{\meas}{\operatorname{meas}}
\newcommand{\spn}{\operatorname{span}}
\newcommand{\dtm}{\operatorname{det}}
%


\newcommand{\overlim}{\mathop{\overline{\lim}}\limits}
\newcommand{\underlim}{\mathop{\underline{\lim}}\limits}
\newcommand{\ttop}[2]{\genfrac{}{}{0pt}{}{#1}{#2}}
\newcommand{\bcu}{\mathop{\txty{\bigcup}}\limits}
\newcommand{\bca}{\mathop{\txty{\bigcap}}\limits}
\newcommand{\bsu}{\mathop{\txty{\sum}}\limits}
\newcommand{\pro}{\mathop{\txty{\prod}}\limits}


\newcommand{\pl}{\partial}
\newcommand{\ptt}{\frac{\pl}{\pl t}}
\newcommand{\ppx}{\frac\pl{\pl x}}
\newcommand{\dds}{\frac d{ds}}
\newcommand{\ddt}{\frac d{dt}}

\newcommand{\intl}{\int\limits}
\newcommand{\iintl}{\iint\limits}
\def\Xint#1{\mathchoice
    {\XXint\displaystyle\textstyle{#1}}%
    {\XXint\textstyle\scriptstyle{#1}}%
    {\XXint\scriptstyle\scriptscriptstyle{#1}}%
    {\XXint\scriptscriptstyle\scriptscriptstyle{#1}}%
    \!\int}
\def\XXint#1#2#3{\setbox0=\hbox{$#1{#2#3}{\int}$}
    \vcenter{\hbox{$#2#3$}}\kern-0.5\wd0}
\def\bint{\Xint-}
\def\dashint{\Xint{\raise4pt\hbox to7pt{\hrulefill}}}
\def\dashiint{\bint\kern-0.15cm\bint}

\newcommand{\ovl}[3]{\int_{#1}^{#2}\kern-#3pt\raise4pt\hbox to7pt{\hrulefill}\ }

\newcommand{\ovll}[3]{\intl_{#1}^{#2}\kern-#3pt\raise4pt\hbox to7pt{\hrulefill}\ }

\newcommand{\tvl}[2]{\iint_{#1}\kern-#2pt\raise4pt\hbox to7pt{\hrulefill}\ }



\newcommand{\omt}{\Om_T}
\newcommand{\plo}{\partial\Omega}
\newcommand{\ovo}{\bar{\Om} }

%
\newcommand{\ci}[1]{C^\infty\!\left({#1}\right)}
\newcommand{\cio}[1]{C_o^\infty\!\left({#1}\right)}
\newcommand{\lloc}[1]{L_{\loc}\!\left({#1}\right)}
\newcommand{\xy}{|x-y|}


\newcommand{\intom}{\intl_{\Om}}
\newcommand{\intbo}{\intl_{\plo}}
\newcommand{\inom}{\int_{\Om}}
\newcommand{\inbo}{\int_{\plo}}
\newcommand{\intrn}{\intl_{\rn}}


\newcommand{\bye}{\end{document}}



%
%
%

%
%
%
%
%

%
%

\input harnack_mono.mac
\title{A Short Proof of H\"older Continuity for Functions in DeGiorgi Classes}
\author{
Colin Klaus\footnote{Author is grateful to the Mathematical Biosciences Institute for partially supporting this research. MBI receives funding through the National Science Foundation grant DMS 1440386}\\
The Mathematical Biosciences Institute\\
Ohio State University\\
Columbus, OH 43210, USA\\
email: {\tt klaus.68@mbi.osu.edu}\\
\and
Naian Liao\footnote{Corresponding author. Supported by NSFC 11701054}\\
College of Mathematics and Statistics\\
Chongqing University\\
Chongqing, 401331, China\\
email: {\tt liaon@cqu.edu.cn}}
\date{}
\maketitle
\begin{abstract}
The goal of this note is to give an alternative proof of local H\"older continuity for functions
in DeGiorgi classes based on an idea of Moser.
\end{abstract}

\section{Introduction and Main Results}
Let $E$ be an open set in $\rr^N$ and $K_{\rho}(y)$ be a cube of edge $2\rho$
centered at $y\in\rr^N$. When $y=0$ we simply write $K_{\rho}$.
The DeGiorgi classes $[DG]_p^{\pm}(E;\gm)$, with some $p>1$
and $\gm>0$, consist of functions $u\in W^{1,p}_{\loc}(E)$ satisfying
\begin{equation}
\int_{K_{\rho}(y)}|D(u-k)_{\pm}|^p\,dx\le\frac{\gm}{(R-\rho)^p}\int_{K_{R}(y)}|(u-k)_{\pm}|^p\,dx
\end{equation}
for any pair of cubes $K_\rho (y)\subset K_{R}(y)\subset E$, and all $k\in\rr$. We further define
\begin{equation*}
[DG]_p(E;\gm)=[DG]_p^+(E;\gm)\cap[DG]_p^-(E;\gm).
\end{equation*}
In the sequel, we refer to the set of parameters $\{p,\,\gm,\,N\}$ as the {\it data}
and use $C$ as a generic constant that can be quantitatively determined {apriori}
only in terms of the data.

For a function $u\in [DG]_p(E;\gm)$ and $K_{2\rho}(y)\subset E$, we set
\begin{equation*}
\mu^+=\essup_{K_{2\rho}(y)}u,\quad\mu^-=\essinf_{K_{2\rho}(y)}u,\quad\om(2\rho)=\essosc_{K_{2\rho}(y)}u=\mu^+-\mu^-.
\end{equation*}
When there is no ambiguity, we scratch ``ess" in the following.

Now we state the following celebrated theorem of DeGiorgi, c.f. \cite{DG, DB, LU}.
\begin{theorem}\label{Thm:1:1}
\begin{enumerate} 
\item[(I)] If $u\in [DG]_p^{\pm}(E;\gm)$, then there is a constant $C>0$ depending only on the data, such that
\begin{equation}
\sup_{K_\rho(y)}(u-k)_{\pm}\le\frac{C}{(R-\rho)^N}\int_{K_R(y)}(u-k)_{\pm}\,dx
\end{equation}
for any pair of cubes $K_\rho(y)\subset K_R(y)\subset E$ and all $k\in\rr$.
\item[(II)] If $u\in [DG]_p(E;\gm)$, there are constants $C>0$ and $0<\al<1$ depending only on the data, such that
for every pair of cubes $K_\rho(y)\subset K_R(y)\subset E$, we have
\begin{equation*}
\om(\rho)\le C\om(R)\bigg(\frac{\rho}{R}\bigg)^{\al}
\end{equation*}
\end{enumerate}
\end{theorem}
Although the DeGiorgi classes were originally modelled after linear elliptic equations with
bounded and measurable coefficients, DeGiorgi's approach to prove local boundedness and local H\"older continuity
of their solutions made no reference to any equation, and
such functional classes are general enough to include local minima or Q-minima of rather general functionals,
which may not admit an Euler equation.
A similar theorem, regarding the local boundedness and local H\"older continuity of solutions to 
quasilinear elliptic equations, was proved by Moser in \cite{Moser}; see also \cite{Nash}, \cite[Chap. 9]{LU}.
However, the original proof of Moser kept referring to the equation.

It has been noted in \cite{DT} that Moser's idea 
 can be employed to show local boundedness of functions in $[DG]_p^+(E;\gm)$, i.e., Part (I) of Theorem \ref{Thm:1:1}.
Thanks to the recent remarks in \cite{DG} on properties of DeGiorgi classes, we are able
 to give an alternative proof of local H\"older continuity for functions
in $[DG]_p(E;\gm)$ based on Moser's idea (\cite{Moser}). This is the main goal of this note. 
For a somewhat analogous approach see also \cite{DMV}.

\

\noi{\it Acknowledgement.} The authors thank Professor Emmanuele DiBenedetto for suggesting the argument
in this note and Professor Ugo Gianazza for discussions and remarks, which greatly
helped to improve the final version of the note.

\section{Some Lemmas}
The generalized DeGiorgi Classes $[GDG]_p^{\pm}(E;\gm)$ are the collection of functions $u\in W^{1,p}_{\loc}(E)$,
for some $p>1$, satisfying
\begin{equation*}
\int_{K_{\rho}(y)}|D(u-k)_{\pm}|^p\,dx
\le\frac{\gm}{(R-\rho)^p}\bigg(\frac{R}{R-\rho}\bigg)^{Np}\int_{K_{R}(y)}|(u-k)_{\pm}|^p\,dx.
\end{equation*}
It is noteworthy that the conclusion of Theorem \ref{Thm:1:1}, Part (I) still holds for functions in $[GDG]_p^{+}(E;\gm)$.
One just has to note that the extra term on the right-hand side of the definition of $[GDG]_p^{+}(E;\gm)$ is 
``homogeneous" with respect to the diameter of the cubes, then one could repeat the proof in
 \cite[Theorem 2.1, Chap. 10]{DB}, \cite[Lemma 2.1]{DT} or \cite[Lemma 5.4, Chap. 2]{LU}.

For ease of notation, we write $\om=\om(2\rho)$. We introduce two functions
that are due to Moser (\cite{Moser}):
\begin{equation*}
w_1=\vp_1(u)=\ln \frac{\om}{2(\mu^+-u)},\quad w_2=\vp_2(u)=\ln \frac{\om}{2(u-\mu^-)}.
\end{equation*}
\begin{lemma}\label{Lm:2:1}
Suppose $u\in[DG](E;\gm)$. 
The two functions $w_{1+}$ and $w_{2+}$ are both in $[GDG]^+_p(E;\bar{\gm})$
for some $\bar{\gamma}$ depending only on the data.
\end{lemma}
\begin{proof} 
Since $\vp_{1+}: (-\infty,\mu^+)\to\rr$ is convex and non-decreasing,
the function $w_{1+}\in [GDG]^+_p(E;\bar{\gm})$ by Lemma 2.1 in \cite{DG}.
Since $\vp_{2+}: (\mu^-,\infty)\to\rr$ is convex, non-increasing and vanishes for $u\ge\mu^-+\frac{\om}{2}$,
the function $w_{2+}\in [GDG]^+_p(E;\bar{\gm})$ by Lemma 2.2 in \cite{DG}.
\end{proof}

\noi Next we need a lemma whose proof can be found on p.356 of \cite{DB} or p.54 of \cite{LU}.
\begin{lemma}\label{Lm:2:2}
Let $v\in W^{1,1}(K_r(y))$ and assume that the set $[v=0]$ has positive measure.
There exists a positive constant $C$ depending only on $N$, such that
\[
\int_{K_{r}(y)}|v|\,dx\le C\frac{r^{N+1}}{|[v=0]|}\int_{K_{ r}(y)}|Dv|\,dx.
\]
\end{lemma}
\noi Finally, we recall Proposition 3.3 from \cite{DG}.
\begin{lemma}\label{Lm:2:3}
Let $u\in [DG]^-_p(E;\gm)$ be  non-negative and bounded above by a positive constant $M$.
Then 
\begin{equation*}
\int_{K_\rho(y)}|D\ln u|^p\,dx\le\frac{\gm p}{(R-\rho)^p}\int_{K_R(y)}\ln\frac{M}{u}\,dx
\end{equation*}
for any pair of cubes $K_{\rho}(y)\subset K_{R}(y)\subset E$.
\end{lemma}
\section{Proof of H\"older Continuity}
Without loss of generality, we may take $y=0$. 
Let us first apply Lemma \ref{Lm:2:3} to $w_1$ and $w_2$.
Indeed, since $u\in [DG]_p(E;\gm)$ we have both $\mu^+-u$
and $u-\mu^-$ members of $[DG]_p^-(E;\gm)$. Therefore, Lemma \ref{Lm:2:3}
yields
\[
\int_{K_{\rho}}\bigg|D\ln \frac{\om}{2(\mu^+-u)}\bigg|^p\,dx
\le\frac{\gm}{(R-\rho)^p}\int_{K_{R}}\ln\frac{\om}{\mu^+-u}\,dx,
\]
that is,
\begin{equation}\label{Eq:3:1}
\int_{K_{\rho}}|Dw_1|^p\,dx\le\frac{\gm}{(R-\rho)^p}\int_{K_{R}}|w_1|\,dx+\gm\frac{R^{N}}{(R-\rho)^p}.
\end{equation}
Similarly, we have
\[\int_{K_{\rho}}|Dw_2|^p\,dx\le\frac{\gm}{(R-\rho)^p}\int_{K_{R}}|w_2|\,dx+\gm\frac{R^{N}}{(R-\rho)^p}.\]

Now we go with two alternatives: either 
\[|[u\le \mu^+-\frac{\om}2]\cap K_{\rho}|\ge\frac12|K_\rho|\quad\text{or}\quad
|[u\ge \mu^-+\frac{\om}2]\cap K_{\rho}|\ge\frac12|K_\rho|.\]
In terms of $w_1$ and $w_2$, this may be rephrased as either
\[|[w_1\le 0]\cap K_\rho|\ge\frac12 |K_\rho|\quad\text{or}\quad|[w_2\le 0]\cap K_\rho|\ge\frac12|K_\rho|.\]
Suppose the first alternative is in force, the second alternative being similar. 
We may employ Lemma \ref{Lm:2:2} and the fact that
$w_1\ge-\ln2$ to obtain that
\begin{align*}
\int_{K_\rho}|w_1|\,dx&=\int_{K_\rho}w_{1+}\,dx+\int_{K_\rho}w_{1-}\,dx\\
&\le C\rho\int_{K_\rho}|Dw_{1+}|\,dx+C\rho^N.\\
\end{align*}
The integral term on the right-hand side is estimated by H\"older's inequality, Young's inequality and \eqref{Eq:3:1} as
\begin{align*}
&C\rho\int_{K_\rho}|Dw_{1+}|\,dx\\
&\le C \rho^{1+N-\frac{N}p}\bigg(\int_{K_\rho}|Dw_{1+}|^p\,dx\bigg)^{\frac1p}\\
&\le C \rho^{1+N-\frac{N}p}\bigg(\frac{\gm}{(R-\rho)^p}\int_{K_{R}}|w_1|\,dx+\gm\frac{R^{N}}{(R-\rho)^p}\bigg)^{\frac1p}\\
&\le C \frac{\rho^{1+N-\frac{N}p}}{R-\rho}\bigg(\int_{K_R}|w_1|\,dx\bigg)^{\frac1p}+
C\frac{\rho^{1+N}}{R-\rho}\bigg(\frac{R}{\rho}\bigg)^{\frac{N}p}.\\
\end{align*}
Thus we obtain
\begin{equation}\label{Eq:3:2}
\int_{K_\rho}|w_1|\,dx\le 
C \frac{\rho^{1+N-\frac{N}p}}{R-\rho}\bigg(\int_{K_R}|w_1|\,dx\bigg)^{\frac1p}+
C\frac{\rho^{1+N}}{R-\rho}\bigg(\frac{R}{\rho}\bigg)^{\frac{N}p}+C\rho^N.
\end{equation}
An interpolation argument (see \cite[Lemma 4.3, Chap. I]{DB-DPE}) yields that
\begin{equation}\label{Eq:3:3}
\frac{1}{\rho^N}\int_{K_{\rho}}|w_1|\,dx\le C(data).
\end{equation}

An application of Lemma \ref{Lm:2:1} gives that $w_{1+}\in [GDG]_p^+(E;\bar{\gm})$.
As a result, Theorem \ref{Thm:1:1}, Part (I) holds for $w_{1+}$. The supreme estimate together
with \eqref{Eq:3:3} yields that
\[
\sup_{K_{\frac{\rho}2}}w_{1+}\le C\,\dashint_{K_{\rho}}w_{1+}\,dx\le C(data),
\]
which implies
\[
\essup_{K_{\frac{\rho}2}}u\le \mu^+-\frac{1}{2e^{C}}\om.
\]
Therefore
\[
\essosc_{K_{\frac{\rho}2}}u\le(1-\frac{1}{2e^{C}})\om.
\]
A standard iteration finishes the proof.

\end{document}